\documentclass{ws-aa}

\newcommand{\R}{\mathbb{R}}

\newcommand{\N}{\mathbb{N}}

\newcommand{\cC}{\mathcal{C}}

\newcommand{\cJ}{\mathcal{J}}

\newcommand{\cL}{\mathcal{L}}
\newcommand{\cP}{\mathcal{P}}

\newcommand{\oP}{\overline{\cP}}

\newcommand{\be}{\begin{equation}}
\newcommand{\ee}{\end{equation}}

\newcommand{\inn}[2]{{\langle #1,#2 \rangle}}
\newcommand{\norm}[1]{\|#1\|}
\newcommand{\wt}[1]{\widetilde{#1}}
\newcommand{\wh}[1]{\widehat{#1}}

\newcommand{\fc}{*_T}
\newcommand{\Y}{\mathbb{Y}}

\begin{document}

\markboth{Navascu\'es, Massopust}{Fractal Convolution}

%\catchline{}{}{}{}{}
%
%%%%%%%%%%%%%%%%%%%%%%%%%%%%%%%%%%%%%%%%%%%%%%%%%%%%%%%%%%%%%%%%%%%%

\title{FRACTAL CONVOLUTION: \\
A NEW OPERATION BETWEEN FUNCTIONS}

\author{MAR\'IA A. NAVASCU\'ES}

\address{Department of Applied Mathematics, University of Zaragoza,\\
 C/ Mar\'ia de Luna 3, Zaragoza 50018, Spain\\
manavas@unizar.es}

\author{PETER MASSOPUST}

\address{Zentrum Mathematik, Technische Universit\"at M\"unchen,\\
Boltzmannstrasse 3, 85748 Garching b. M\"unchen, Germany\\
massopust@ma.tum.de}

\maketitle

%\begin{history}
%\received{(Day Month Year)}
%\revised{(Day Month Year)}
%%\accepted{(Day Month Year)}
%%\comby{(xxxxxxxxxx)}
%\end{history}

\noindent
This paper is now published (in revised form) in Fract. Calc. Appl. Anal. Vol. 22, No 3 (2019), pp. 619--643,  DOI: 10.1515/fca-2019-0035, and is available online at http://www.degruyter.com/view/j/fca
\vspace*{20pt}
\begin{abstract}
In this paper we define an internal binary operation between functions called in the text \emph{fractal convolution}, that applies a pair of mappings into a fractal function. This is done by means of a suitable Iterated Function System. We study in detail the operation in $\mathcal{L}^p$ spaces and in sets of continuous functions, in a different way to previous works of the authors. We develop some properties of the operation and its associated sets. The lateral convolutions with the null function provide linear operators whose characteristics are explored. The last part of the article deals with the construction of convolved fractals bases and frames in Banach and Hilbert spaces of functions.
\end{abstract}



%%%%%%%%%%%%%%%%%%%%%%%%%%
\section{Introduction}\label{sec:introduction}
%%%%%%%%%%%%%%%%%%%%%%%%%%

As part of his work, B. Mandelbrot defined stochastic fractal and fractional functions, opening the way to new scenarios to model and process complex experimental signals and to novel theoretical and mathematical tools. (See, for instance, \cite{1} and \cite{2}.) Later, several fractal maps were constructed by means of affinities whose successive application to sets in the Euclidean plane led to fractal curves. There exists an extensive bibliography on this topic (see, for instance, the albeit incomplete list of references \cite{3, 4, 5, 6, 7, 8, 9}) although the concept is already present in the work of Mandelbrot.

These initial constructions were later extended to more general iterating systems involving nonlinear mappings in the Euclidean plane. In previous references (for instance, \cite{10}, \cite{11}), we defined fractal maps that are close to standard classical functions such as polynomials and splines. They may share or preserve some properties of the latter or display new characteristics such as non-smoothness or quasi-random behavior. These fractal maps are defined by means of iterated function systems (IFS) which involve a certain combination of two functions (aside from other elements such as scale factors) that gives rise to a new class of fractal functions \cite{12}. In this way, a wide spectrum of fractal functions is constructed that may have properties such as integrability, nowhere continuity, or smoothness with a prefixed order of regularity. If the scale functions are chosen accordingly, these fractal functions are close to the original standard function mimicking their behavior.

In this paper, we resume a  binary operation between functions, called \emph{fractal convolution} in the reference \cite{12}, which maps a pair of functions to a fractal function exploiting the latter's dependence on what are called seed and base functions. We develop some properties of this binary operation which is -- unlike ordinary convolution -- asymmetric, i.e., non-commutative. Fractal convolution, as we show, possesses great versatility when applied in Banach and Hilbert spaces of functions.

The outline of this paper is as follows. In Section 2, fractal convolution is defined on very general spaces of functions. Section 3 deals with the corresponding product operator and investigates some properties of fractal convolution. Section 4 studies the special case of continuous maps, however, in a very different way than the previous work of the authors \cite{11}. In Section 5, sets of convolution are introduced and in Sections 6 and 7, we apply fractal convolution to the construction of bases and frames in Banach and Hilbert spaces of functions.

%%%%%%%%%%%%%%%%%%%%%%%%%%
\section{A Fractal-type Operation of Fractal Functions}
%%%%%%%%%%%%%%%%%%%%%%%%%%

Throughout this article, we use the symbol $\N_N :=\{1, \ldots, N\}$ to denote the initial segment of length $N$ of $\N$, $N\geq 2$.

Let $X$ be a nonempty set and suppose $\{L_n : X\to X : n \in \N_N\}$ is a family of bijections with the property that $\{L_n(X) : n \in\N_N\}$ forms a
partition $\Delta$ of $X$, i.e., \begin{equation}\label{p} X = \bigcup_{n=1}^N L_n (X), \quad  L_n(X)\cap L_m(X) = \emptyset,\quad\forall n\neq m\in \N_N. \end{equation}
\noindent Now suppose that $(\Y,d_\Y)$ is a complete metric space with metric $d_\Y$. Recall that a mapping $f:X\to \Y$ is called bounded (with respect
to the metric $d_\Y$) if there exists an $M> 0$ so that for all $x_1, x_2\in X$, $d_\Y(f(x_1),f(x_2)) < M$.
Denote by $B(X, \Y)$ the set
\[
B(X, \Y) := \{f : X\to \Y : \quad f \quad \textrm{bounded}\}.
\]
Endowed with the metric
\[
d(f,g): = \displaystyle{\sup_{x\in X}} \,d_\Y(f(x), g(x)),
\]
$(B(X, \Y), d)$ becomes a complete metric space.

Under the usual addition and scalar multiplication of functions, the space $B(X,\Y)$ becomes a metric linear space \cite{13}. Recall that a \emph{metric
linear space} is a vector space endowed with a metric under which the operations of vector addition and scalar multiplication become continuous.

For $n \in \N_N$, let $v_n: X\times \Y \to \Y$ be a mapping that is uniformly contractive in the second variable, i.e., there exists an $\ell\in [0,1)$
so that for all $y_1, y_2\in \Y$ \be\label{scon} d_\Y (v_n(x, y_1), v_n(x, y_2)) \leq \ell\, d_\Y (y_1, y_2), \quad\forall x\in X. \ee Define an
operator $T: B(X,\Y)\to \Y^{X}$ by \be\label{RB} T g (x) := \sum_{n=1}^N v_n (L_n^{-1} (x), g \circ L_n^{-1} (x))\,\chi_{L_n(X)}(x), \ee where
$\chi_M$ denotes the characteristic function of a set $M$. Note that $T$ is well-defined and since $f$ is bounded and each $v_n$ contractive in its second variable, $T f\in B(X,\Y)$. The operator $T$ is usually referred to as a Read--Bajraktarevi\'c (RB) operator. (See \cite{14, 15}).

Moreover, by (\ref{scon}), we obtain for all $g,h \in B(X, \Y)$ the following inequality
%\begin{equation}
\begin{eqnarray}\label{estim}
d(T g, T h) & =  & \sup_{x\in X} d_\Y (T g (x), T h (x))\nonumber\\
& = & \sup_{x\in X} d_\Y (v(L_n^{-1} (x), g(L_n^{-1} (x))), v(L_n^{-1} (x), h(L_n^{-1} (x))))\nonumber\\
& \leq & \ell \sup_{x\in X} d_\Y (g\circ L_n^{-1} (x), h \circ L_n^{-1} (x)) \leq \ell\, d_\Y(g,h).\nonumber
\end{eqnarray}
%\end{equation}
(To simplify notation, we set $v(x,y):= \sum\limits_{n=1}^N v_n (x, y)\,\chi_{L_n(X)}(x)$ in the above equation.) In other words, $T$ is a
contraction on the complete metric space $B(X,\Y)$ and, by the Banach Fixed Point Theorem, has therefore a unique fixed point $\wt{f}$ in $B(X,\Y)$.
This unique fixed point is called a \emph{fractal function} or a \emph{function attractor of $T$}.

Next, we would like to consider a special choice of mappings $v_n$. To this end, we require the concept of an $F$-space. We recall that a metric
$d:\Y\times\Y\to \R$ is called \emph{complete} if every Cauchy sequence in $\Y$ converges with respect to $d$ to a point of $\Y$ and
\emph{translation-invariant} if $d(x+a,y+a) = d(x,y)$, for all $x,y,a\in \Y$.

\begin{definition}[\cite{13}]
A topological vector space $\Y$ is called an \emph{${F}$-space}  if its topology is induced by a complete translation-invariant metric $d$.
\end{definition}

Now suppose that $\Y$ is an $F$-space. Denote its metric by $d_\Y$. We define mappings $v_n:X\times\Y\to \Y$ by \be\label{specialv} v_n (x,y) := q_n
(x) + \alpha_n (x) \,y,\quad n \in \N_N, \ee where $q_n \in B(X,\Y)$ and $\alpha_n : X\to \R$ is a function such that $\Vert \alpha_n \Vert_\infty < 1$, called a \emph{scale function}.

If in addition we require that the metric $d_\Y$ is \emph{homogeneous}, that is,
\[
d_\Y(\lambda y_1, \lambda y_2) = |\lambda|\, d_\Y(y_1,y_2), \quad \forall \lambda\in\R, \quad \forall y_1, y_2\in \Y,
\]
then $v_n$ given by (\ref{specialv}) satisfies condition (\ref{scon}) provided that the functions $\alpha_n$ are bounded on $X$ with bounds in $[0,1)$.
For then
\begin{eqnarray}
d_\Y (q_n (x) + \alpha_n (x) \,y_1,q_n (x) + \alpha_n (x) \,y_2) &= & d_\Y(\alpha_n (x) \,y_1,\alpha_n (x) \,y_2)\nonumber \\ & =&  |\alpha_n(x)| d_\Y (y_1, y_2)\nonumber\\
& \leq & \Lambda \,d_\Y (y_1, y_2).\nonumber
\end{eqnarray}
Here, we set $\alpha(x):=(\alpha_n(x))_{n=1}^N$ and $\Lambda := \sup\{\|\alpha_n\|_{\infty}: x\in X;\,n\in \N_N\}$. Note that $\Lambda < 1$.

The associated RB operator (\ref{RB}) is now affine and has the form
\[
T g = \sum_{n=1}^N q_n\circ L_n^{-1} \,\chi_{L_n(X)} + \sum_{n=1}^N (\alpha_n\circ L_n^{-1})\cdot (g\circ L_n^{-1})\,\chi_{L_n(X)},
\]
or, equivalently,
\[
T g = q_n\circ L_n^{-1} + (\alpha_n\circ L_n^{-1}) \cdot (g\circ L_n^{-1}), \quad \textrm{on} \quad L_n(X), n\in \N_N.
\]

The next result is a special case of Theorem 1 in \cite{16} and its validity follows directly from the above considerations.

\begin{theorem}\label{exfix}
Let $\Y$ be an $F$-space with homogeneous metric $d_\Y$. Let $X$ be a nonempty set. Suppose that $\{L_n : X\to X\}_{n \in \N_N}$ is a family of
bijections satisfying property (\ref{p}).

Define a mapping $T:  B(X,\Y)\to B(X,\Y)$ by \be\label{eq3.4} T g = \sum_{n=1}^N (q_n\circ L_n^{-1}) \,\chi_{L_n(X)} + \sum_{n=1}^N (\alpha_n\circ
L_n^{-1})\cdot (g\circ L_n^{-1})\,\chi_{L_n(X)} \ee If $\Lambda < 1$ then the operator $T$ is contractive on the complete metric space
$B(X, \Y)$ and its unique fixed point $\wt{f}$ satisfies the self-referential equation \be\label{3.4} \wt{f} = (q_n\circ L_n^{-1}) + (\alpha_n\circ
L_n^{-1}) \cdot (\wt{f}\circ L_n^{-1}), \quad \textrm{on} \quad L_n(X), n\in \N_N. \ee
\end{theorem}
Choose arbitrary $f, b\in B(X,\Y)$ and set
\[
q_n := f\circ L_n - \alpha_n\cdot b.
\]
Then the RB operator $T$ becomes \be Tg = f + (\alpha_n\circ L_n^{-1})\cdot (g-b)\circ L_n^{-1}, \quad \textrm{on} \quad L_n(X), n\in \N_N. \ee and, under the
assumption that $\Lambda < 1$ its unique fixed point $\wt{f}$ satisfies the self-referential equation
\[
\wt{f} = f + (\alpha_n \circ L_n^{-1})\cdot (\wt{f}-b)\circ L_n^{-1}, \quad \textrm{on} \quad L_n(X), n\in \N_N.
\]

On the set $B(X,\Y)\times B(X,\Y)$, we define a binary operation $\cP := \cP_{\alpha, \Delta}$ with respect to a given partition $\Delta$ and a given scale vector  $\alpha $ of contractive maps by (cf. \cite{12})

\be \cP(f,b) := \wt{f}. \ee Instead of $\cP (f,b)$, we write for short $f \fc b$. The binary operation
\[
\fc: B(X,\Y)\times B(X,\Y)\to B(X,\Y)
\]
will be called the \emph{fractal convolution} of the \emph{seed function} $f$ with the \emph{base function} $b$ (for the RB operator $T$).

The fractal convolution of $f$ with $b$ satisfies the fixed point equation
\be\label{fixp} (f\fc b) = f + (\alpha_n \circ L_n^{-1}) ((f\fc b) - b)\circ L_n^{-1}\quad \textrm{on} \quad
L_n(X), n\in \N_N. \ee
This is a self-referential equation endowing the graph of $f\fc b$ with a fractal structure.

\begin{theorem}
The mapping $B(X,\Y)\times B(X,\Y)\ni (f,b) \mapsto f\fc b \in B(X,\Y)$ is linear and bounded.
\end{theorem}
\begin{proof}
The result is an immediate consequence of Proposition 2.3. in \cite{12} applied to the current slightly more general setting.
\end{proof}

\begin{remark}\label{rem:110518}
The binary operation $\fc$ is in general not commutative. It is, however, idempotent since $(f\fc f)=f$ for any $f\in B(X,\Y).$
\end{remark}

%%%%%%%%%%%%%%%%%%%%%%%%%%
\section{Fractal Convolution on the Lebesgue Spaces $\mathcal{L}^p$}
%%%%%%%%%%%%%%%%%%%%%%%%%%

In this section, we consider the fractal convolution associated with an RB operator $T$ acting on the Lebesgue spaces $\mathcal{L}^p$ with $0 < p \leq \infty$.
For this purpose, we take the set-up in the previous section and choose for $X$ a nonempty interval $I :=[x_0,x_N]$ on $\R$. Further, we set $(\Y,d_\Y):=(\R,
|\cdot|)$ with the usual norm on $\R$.

Recall that the Lebesgue spaces $\mathcal{L}^p (I)$ are Banach spaces with norm
\[
\norm{g}_p := \left(\int_I |g(x)|^p dx\right)^{1/p}
\]
for $1\leq p \leq \infty$, and $F$-spaces whose topology is induced by the complete translation invariant metric \be d_p (g,h) := \norm{g-h}_{p}^p =
\int_I |g(x) - h(x)|^p dx, \ee for $0 < p < 1$.  To facilitate notation, we will use the norm symbol $\Vert \cdot \Vert_p$
also in the case $0<p<1$  with the obvious interpretation when no confusion can arise.

Let $\Delta: x_0<x_1<...<x_N$ be any partition of the interval $I$ and, for all $n\in \N_N$, let $L_n(x)=a_n x+b_n$ be fixed contractive affinities such that
$L_n(x_0)=x_{n-1}$ and $L_n(x_N)=x_n$.

Fix a partition $\Delta$ and functions $f, b\in \mathcal{L}^p (I)$, and a scale vector function $\alpha\in (\cL^\infty (I))^N$ satisfying $\Lambda < 1$ where
\[
\Lambda := esssup\{|\alpha_n(x)| : x\in I, n\in\N\}.
\]
As above, we associate with $f, b, \alpha$, and $\Delta$ an operator $T := T_{f,b,\alpha, \Delta}:\mathcal{L}^p(I)\to \mathcal{L}^p(I)$ as follows: \be\label{eq1}
Tg(x):=f(x) + (\alpha_n\circ L_n^{-1})(x)\cdot (g-b)\circ L_n^{-1}(x), \ee for $x\in I_n$, $n\in \N_N$. The intervals $I_n$ are defined as $I_n
:=(x_{n-1},x_n]$ for $n=2,\ldots,N$, and $I_1=[x_{0},x_1]$.

It follows from Theorem ~\ref{exfix} that the operator $T$ is a contraction on $\mathcal{L}^p(I)$ admitting a unique fixed point $\widetilde{f}\in \mathcal{L}^p(I)$.

As before, we define on the set $\mathcal{L}^p(I)\times \mathcal{L}^p(I)$ a binary operation $\cP = \cP_{\Delta, \alpha}$ with respect to a given fixed partition $\Delta$ and a
given contractive scale vector function $\alpha\in (\cL^\infty(I))^N$ by \be \cP(f,b) := f\fc b := \wt{f}, \ee called as before the fractal convolution of $f\in  \mathcal{L}^p(I)$ with $b\in  \mathcal{L}^p(I)$.

The map $f \fc b$ is in general discontinuous and not interpolatory with respect to $f$ on $\Delta$; see Figure \ref{fig1}.

\begin{figure}[h]
\includegraphics[angle=0, width=1\textwidth]{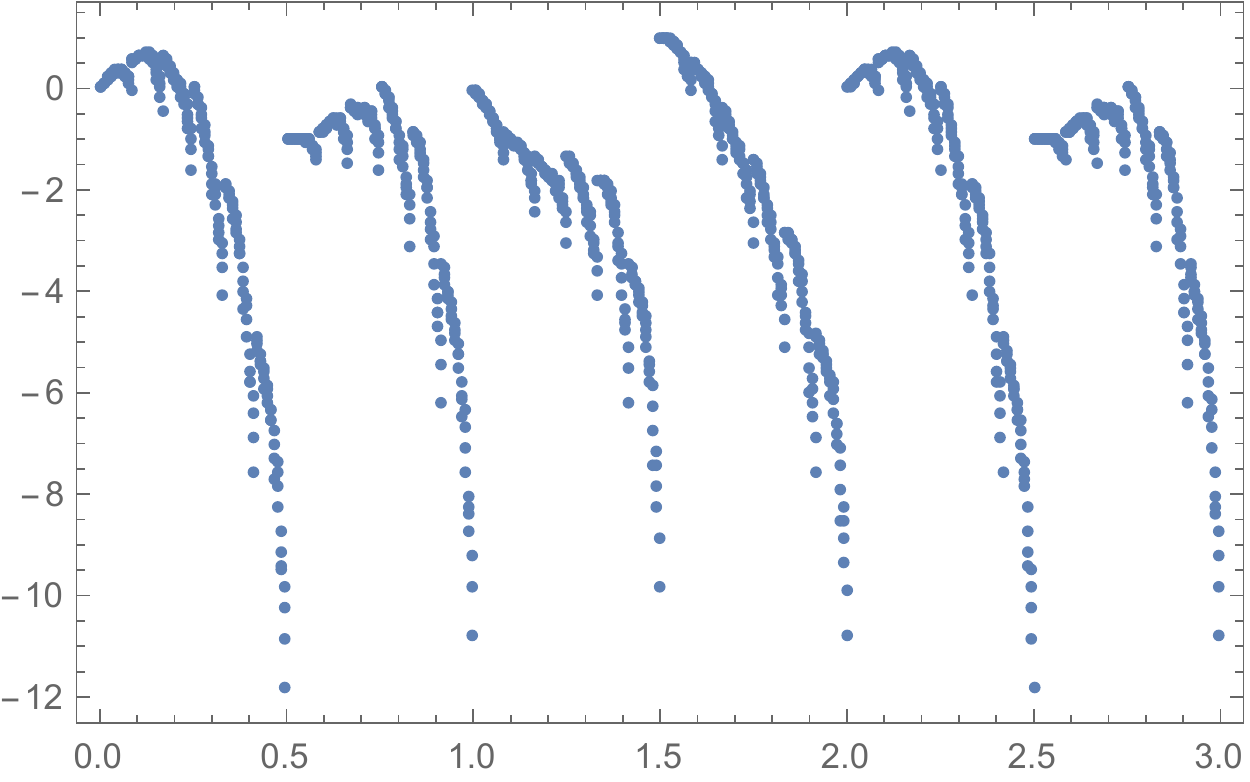} \caption{\small Graph of the fractal convolution $\sin(3 \pi x)*_T \exp(x)$
with respect to a uniform partition of the interval $I=[0,3]$ with $N=6$ and scale functions $\alpha_n(x)=x/8$ for $n=1, 2, \ldots, 6.$}\label{fig1}
\end{figure}

If one wishes to approximate $f\fc b$ to within $\epsilon>0$ of $f$ at the nodes $x_n$, one must choose $b$ close to $f$ at the endpoints of the
interval according to the following result (cf. \cite{12}).

\begin{proposition} The distance between $f$ and its fractal convolution $f\fc b$ does not exceed $\epsilon$ on the nodes of the partition, i.e.,
\[
\sup\{\vert (f \fc b)(x_n) -f(x_n) \vert : n=0, 1, 2,\ldots, N\}  \leq \epsilon,
\]
provided
\[
\max \{\vert f(x_0)-b(x_0)\vert, \vert f(x_N)-b(x_N)\vert\}\le \epsilon(1-\Lambda).
\]
\end{proposition}
The next result can be found in, for instance, \cite[Proposition 3]{11} for $1\leq p \leq \infty$, and \cite[Note, p. 253]{11} for $0 < p <1$ .
\begin{theorem}\label{prop3.2}
The $\mathcal{L}^p$-distance between $f\fc b$ and $f$ obeys the following estimates:
\begin{equation}\label{eq5}
\Vert (f\fc b) -f \Vert_p  \leq \frac{\Lambda}{1-\Lambda} \Vert f-b\Vert_p, \quad 1\leq p \leq \infty,
\end{equation}
and

\be\label{eq5a}
d_p (f\fc b), f) \leq \frac{\Lambda^p}{1-\Lambda^p} d_p(f, b), \quad 0 < p < 1.
\ee

\end{theorem}
%\begin{proof}
%\begin{eqnarray}
%\norm{(f\fc b) - f}_p &=\max \norm{(\alpha_n \circ L_n^{-1}) ((f\fc b) - b)\circ L_n^{-1}}_p\\ &\leq \nalpha \norm{(f\fc b) - b}_p \leq \nalpha %\left(\norm{(f \fc b)
%- f}_p + \norm{f - b}_p\right)
%\end{eqnarray}
%Upon transposing the last inequality one arrives at the statement.
%\end{proof}
Note that the inequalities (\ref{eq5}) and (\ref{eq5a}) immediately imply that if either $\alpha$ is the null vector function or $f=b$ then $(f\fc b)=f$.

Endow the product space $\mathcal{L}^p(I)\times \mathcal{L}^p(I)$ with the following norm, respectively, metric
\begin{eqnarray}\label{prodnorm}
%\begin{align}
\norm{(f,b)} &:= & \norm{f}_p + \norm{b}_p, \quad 1 \leq p \leq \infty,\\
d((f,b),(f',b')) &:= & d_p (f,f') + d_p(b,b'), \quad 0 < p < 1.
%\end{align}
\end{eqnarray}
With this norm, respectively, metric, $\mathcal{L}^p(I)\times \mathcal{L}^p(I)$ is complete for $p>0$.
%The next result can be found in reference \cite{jbas}.
\begin{theorem}
The operator $\cP: \mathcal{L}^p(I)\times \mathcal{L}^p(I)\to \mathcal{L}^p(I)$ is linear and bounded.
\end{theorem}
\begin{proof}

The linearity of $\cP$, namely
\begin{eqnarray}
(\gamma f+f')\fc (\gamma b+b')  &= & (\gamma f)\fc (\gamma b) + f'\fc b' \nonumber \\
& = & \gamma (f\fc b) + f'\fc b', \quad \gamma\in \R, \nonumber
\end{eqnarray}
follows from the linearity of $T$ and the uniqueness of the fixed point $\wt{f} = f\fc b$.

For $1\leq p\leq \infty$, the boundedness is implied by the following
argument:
\begin{eqnarray}
\norm{\cP(f,b)}_p &= &\norm{f\fc b}_p = \norm{f \fc b - f + f}_p \leq \norm{f \fc b -f} + \norm{f}_p \nonumber \\
& \leq & \frac{\Lambda}{1 - \Lambda}\,\norm{f-b}_p + \norm{f}_p\\ & \leq & \frac{\Lambda}{1 - \Lambda}\,(\norm{f}_p + \norm{b}_p) + \norm{f}_p \nonumber\\
& = & \frac{1}{1-\Lambda}\norm{f}_p + \frac{\Lambda}{1-\Lambda}\norm{b}_p \nonumber \\
& \leq &\frac{1}{1-\Lambda}\left(\norm{f}_p + \norm{b}_p\right) = \frac{1}{1-\Lambda} \norm{(f,b)}.\nonumber
\end{eqnarray}
Therefore, the norm of $\cP$ satisfies the inequality (see also \cite{12})
 \be \norm{\cP} \leq \frac{1}{1-\Lambda}, \ee
 where the operator norm of $\cP$ is taken with respect to the product norm (\ref{prodnorm}).

Now suppose $0<p<1$. Using the fact that $\cP$ is linear, we have that
\begin{eqnarray}
d_p (\cP(f,b),\cP(f',b')) &= & \norm{f\fc b - f'\fc b'}_p^p = \norm{(f - f')\fc (b - b')}_p^p \nonumber \\
& = & \norm{f''\fc b''}_p^p \leq \frac{\Lambda^p}{1-\Lambda^p}\norm{f'' - b''}_p^p + \norm{f''}_p^p \nonumber\\
& \leq & \frac{\Lambda^p}{1-\Lambda^p}\norm{f''}_p^p + \frac{\Lambda^p}{1-\Lambda^p}\norm{b''}_p^p + \norm{f''}_p^p \nonumber\\
& \leq & \frac{1}{1-\Lambda^p} \left(\norm{f''}_p^p + \norm{b''}\right)_p^p = \frac{1}{1-\Lambda^p} \left(\norm{f - f'}_p^p + \norm{b - b'}_p^p\right) \nonumber\\
& = & \frac{1}{1-\Lambda^p} d_p ((f,b),(f',b')). \nonumber
\end{eqnarray}

\end{proof}

\begin{theorem}\label{prop:051018}
Let $f, f'\in \mathcal{L}^p(I)$ be given seed functions. Then
\begin{equation} \label{eq:140518}
\norm{(f \fc b) - (f' \fc b)}_p \leq \frac{1}{1 - \Lambda} \norm{f - f'}_p,
\end{equation}
for all base
functions $b\in \mathcal{L}^p(I)$.\\
Similarly, for given base functions $b, b'\in \mathcal{L}^p(I)$,
\begin{equation} \label{eq:140519}
\norm{(f \fc b) - (f \fc b')}_p \leq \frac{\Lambda}{1 - \Lambda} \norm{b - b'}_p,
\end{equation}
for
all seed functions $f\in \mathcal{L}^p(I)$.
\end{theorem}
\begin{proof}
Let $f, f'\in \mathcal{L}^p(I)$ be seed functions and $b\in \mathcal{L}^p(I)$ any base function. Then,
\begin{eqnarray}
\norm{(f \fc b) - (f'\fc b)}_p - \norm{f - f'}_p & \leq & \norm{((f \fc b) - (f' \fc b)) - (f - f')}_p \nonumber \\ &  \leq & \Lambda \norm{(f \fc b) - (f' \fc b)}_p.\nonumber
\end{eqnarray}
Above, we used the fixed point equation (\ref{fixp}) for $(f\fc b)$, a change of variable, as well as the fact that $\sum\limits_{n=1}^N a_n = 1$.

The second inequality is proven in a similar way.
\end{proof}

As a corollary to Theorem \ref{prop:051018} one obtains the following result.
\begin{corollary}
Let $f, f', b, b'\in \mathcal{L}^p(I)$. Then \be \norm{(f \fc b) - (f' \fc b')}_p \leq \frac{1}{1 - \Lambda}\left(\norm{f-f'}_p + \Lambda \norm{b-b'}_p\right). \ee
\end{corollary}
\begin{proof}
\begin{eqnarray}
\norm{(f \fc b) - (f' \fc b')}_p &= & \norm{(f \fc b) - (f' \fc b) + (f'\fc b) - (f'\fc b')}_p\nonumber \\ &\leq &\norm{(f \fc b) - (f' \fc b)}_p + \norm{(f'\fc b) - (f'\fc b')}_p \nonumber\\
&\leq & \frac{1}{1-\Lambda}\norm{f-f'}_p + \frac{\Lambda}{1-\Lambda}\norm{b-b'}_p \nonumber \\ & = & \frac{1}{1-\Lambda}\left(\norm{f-f'}_p + \Lambda
\norm{b-b'}_p\right),\nonumber
\end{eqnarray}
with the modification $\Lambda \to \Lambda^p$ if $0 < p < 1$ (See Theorem \ref{prop3.2}.)
\end{proof}

We can also define a partial fractal convolutions where one of the inputs of $\cP$ is held fixed. To this end, let
\begin{eqnarray}
\cP_f^1 (b) := \cP(f,b):\mathcal{L}^p(I)\to \mathcal{L}^p(I),\\ b\mapsto f\fc b,
\end{eqnarray}
for a fixed $f\in \mathcal{L}^p(I)$, and
\begin{eqnarray}
\cP_b^2 (f) := \cP(f,b):\mathcal{L}^p(I)\to \mathcal{L}^p(I),\\ f\mapsto f\fc b,
\end{eqnarray}
for a fixed $b\in \mathcal{L}^p(I)$.

By Theorem \ref{prop:051018}, the partial fractal convolution operators are nonlinear and Lipschitz continuous, and they satisfy the following inequalities:
\begin{eqnarray}
\norm{\cP_f^1 (b) - \cP_f^1 (b')}_p &\leq & \frac{\Lambda}{1 - \Lambda} \norm{b - b'}_p,\\ \norm{\cP_b^2 (f) - \cP_b^2 (f')}_p &\leq & \frac{1}{1 - \Lambda}
\norm{f - f'}_p,
\end{eqnarray}
(with suitable modifications for $0<p<1$).

\begin{remark}
Note that $\cP_f^1$ is a contractive operator on $E$ if $\Lambda < \frac{1}{2}$. Its unique fixed point $b\in \mathcal{L}^p (I)$ satisfies $P_f^1(b) = f\fc b = b$,
which by the uniqueness of the fixed point $\wt{f}$ of $T$, implies that $b = f$. (See Remark \ref{rem:110519}.)
\end{remark}

%%%%%%%%%%%%%%%%%%%%%%%%%%
\section{Fractal convolution of continuous functions}
%%%%%%%%%%%%%%%%%%%%%%%%%%

%In this type of IFS, besides a set and a measure attractors, there %is a function attractor $\widehat{f}$, and the action of the IFS can %be
%identified with the action of an operator $T^{\alpha}$ on a set %of continuous functions, as we will see in this subsection.

Let us consider a set of data $D=\{(x_n, y_n)\}_{n=0}^N$ and a nonempty interval $I :=[x_0, x_N]$. Denote by $\cC(I)$ the Banach space of continuous functions $I\to \R$ endowed with the uniform (or supremum) norm.

\begin{definition}
Denote by $\mathcal{J}_D$ the set of all continuous interpolants of $D=\{(x_n, y_n)\}_{n=0}^N$, i.e.,
\begin{equation}
\mathcal{J}_D := \{f\in \cC(I): f(x_n)=y_n; \;n=0,1,\ldots, N\}.
\end{equation}
\end{definition}

\begin{proposition}
$\mathcal{J}_D$ is a complete metric subspace (with respect to the metric induced by the supremum norm) and a convex subset of $\cC(I)$.
 \end{proposition}

 \begin{proof}
Let us consider  a convergent sequence of elements  $\{f_m\}\subset\mathcal{J}_D\subset \cC(I)$ such that $\lim\limits_{m\to\infty} f_m=f $ where the limit is taken with respect to the metric induced by the supremum norm. By uniform convergence, $f\in \cC(I)$ and $f(x_n) = \lim\limits_{m\to\infty} f_m(x_n) =
\lim\limits_{m\to\infty} y_n=y_n$, for all $n = 0,1,\ldots, N$. Thus, $f\in \mathcal{J}_D$. Consequently, $\mathcal{J}_D$ is a closed subset of the
complete metric space $\cC(I)$ and therefore complete. The convexity is a straightforward consequence of the definition of $\mathcal{J}_D$.
 \end{proof}

For $n\in \N_N$, let $\alpha_n, q_n\in \cC(I)$. The operator $\widehat{T}:\mathcal{J}_D \to \mathcal{J}_D$ defined by %\[\widehat{T}f(t)=F_n(L_n^{-1}(t),f \circ %L_n^{-1}(t))
%\ \ \ \forall %\ t \in [x_{n-1},x_n],\ \ n=1,2,...,N \] or
\begin{equation}\label{eq:11111}
\widehat{T}f(x) := (\alpha_n \circ L_n^{-1})(x) \cdot (f \circ L_n^{-1})(x)+q_n\circ L_n^{-1}(x),\quad\forall \ x \in [x_{n-1},x_n],\ \ n\in \N_N,
\end{equation}
satisfying the join-up conditions
%\begin{eqnarray}
$$(\wh{T} f)(x_0) = y_0,  \quad  (\wh{T} f)(x_N) = y_N $$
 $$ (\wh{T} f)(x_n-) = y_n = (\wh{T} f)(x_n+), \quad n = 1, \ldots, N-1,$$
%\end{eqnarray}
maps an interpolant of $D$ into another interpolant of $D$ and it can be considered as a discrete dynamical system on $\mathcal{J}_D$. If $\Lambda <1$ then  $\widehat{T}$ is contractive and possesses a unique fixed point $\widehat{f}$.

%We %write the exponent $\alpha$ in order to display the dependence of %the transformation with respect to this parameter.

The function attractor $\widehat{f}$ depends continuously on the scale vector of functions $\alpha := (\alpha_1, \ldots, \alpha_n)$. This is a consequence of the continuous dependence of fixed points on parameters
for contractive mappings. (See, for instance, \cite{17}).

The next result shows that the operator $\widehat{T}$ can be extended to the complete metric space of continuous functions $\cC(I)$.

\begin{theorem}[\cite{18}]\label{th2}
Let $M$ be a nonempty closed subset of a metric space $X$. Let $T: M\subset X\rightarrow Y$ be a continuous operator to a normed space $Y$. Then $T$
has a continuous extension $\overline{T}: X\rightarrow conv T(M)$, where $conv T(M)$ denotes the convex hull of $T(M)$.
\end{theorem}

\begin{theorem} \label{th3}
The operator $\widehat{T}:\mathcal{J}_D \subset \cC(I)\rightarrow \mathcal{J}_D$ can be extended to a continuous operator $\overline{T}:\cC(I) \rightarrow\mathcal{J}_D$.
\end{theorem}

\begin{proof}
$\widehat{T}$ is continuous because it is a contraction. In Theorem \ref{th2}, set $M:=\mathcal{J}_D$ and $X := Y := \cC(I)$. Then $T=\widehat{T}$ admits
a continuous extension
\[
\overline{T}:\cC(I) \rightarrow conv \widehat{T}(\mathcal{J}_D).
\]
As $\widehat{T}(\mathcal{J}_D) \subset\mathcal{J}_D$ and $\mathcal{J}_D$ is convex, $conv \widehat{T}(\mathcal{J}_D) = \mathcal{J}_D$.
\end{proof}

Let us now consider the RB operator (\ref{eq1}), i.e., we choose
\[
q_n(x) := f\circ L_n(x) - \alpha_n(x) b(x),
\]
where $f, b \in \mathcal{J}_D$ are two fixed continuous functions interpolating the set of data $D$.

For this choice of functions, the operator $\wh{T}$ satisfies the conditions needed to obtain a continuous fractal function. If the scale vector function $\alpha$ satisfies
$\Lambda < 1$, then the fixed point equation for $\wh{T}$ reads \be \wt{f}(x) = f(x)+ (\alpha_n \circ L_n^{-1})(x) \cdot (\wt{f}-b) \circ
L_n^{-1} (x),\quad x \in I_n, \;\; n\in \N_N, \ee
where $I_n=[x_{n-1}, x_n]$. The fixed point $\wt{f}$ is the fractal convolution of $f$ with $b$. \\

But we wish now to extend the convolution to any pair of continuous functions $f', b'.$  To this end, consider the operator $\cP: \cJ_D\times\cJ_D\to \cJ_D$ with $\cP (f,b) := f\fc b$. By Theorem \ref{th3}, this operator can be continuously extended to an operator \be \oP:
\cC(I)\times \cC(I)\to \cJ_D. \ee

%and %\[q_n(t)= f\circ L_n(t) - \alpha(t) b(t).\] %The Iterated Function System satisfies the join-up conditions because $f$ and $b$ are both
%interpolants of the data. %In particular they agree at the extremes of the interval. \\ %The fixed point equation associated with this system %is:
%\[\widehat{T}(g)= f(t)+ \alpha(t) (g-b) \circ L_n^{-1} (t),\] %for $t \in I_n$. The fixed point of $T$ is the fractal convolution of $f$ and %$b$. Let
%us %consider %\[\mathcal{P}: \mathcal{J}_D \times \mathcal{J}_D \rightarrow \mathcal{J}_D,\] %such that $\mathcal{P}(f,b)=f*b$. Using the %arguments of
%Theorems \ref{th:140318}, \ref{th:140319}, %this operator can be extended to %\[\overline{\mathcal{P}}: \mathcal{C}[a,b] \times %\mathcal{C}[a,b]
%\rightarrow \mathcal{J}_D.\]

Thus, the fractal convolution of two arbitrary continuous functions $f', b'$, not necessarily matching any join-up conditions at the nodes of $I$, can be defined as:
\[
f' \fc b' := \overline{\mathcal{P}}(f',b').
\]
The result of this binary operation is a continuous interpolant of the set of data $D.$

\section{Sets Associated with Fractal Convolution}

%In the following we will consider  $\Lambda:=\max \{\Vert \alpha_n \Vert_\infty \}.$

Given subsets $F, B\subset \mathcal{L}^p(I)$, let us define the following convolution sets:
\[
F\fc b :=\{ f\fc b: f\in F \},
\]
\[
f\fc B :=\{f\fc b:  b\in B \},
\]
and
\[
F\fc B :=\{f\fc b: f\in F \wedge b\in B \}.
\]
From the definitions we deduce the following properties of the convolution sets:
\begin{enumerate}
\item[(1)] If $b\in F$ then $b\in F*_T b,$ since $b=b\fc b$ (Remark \ref{rem:110518}).
\item[(2)] If $f\in B$ then $f\in f\fc B.$ %\item If $F\cap B \neq \emptyset$, $(F\cup B)
  %  \subseteq F*B$.
\end{enumerate}
Consequently,
\begin{enumerate}
\item[(3)] $(F\cap B) \subseteq (F\fc B)\bigcap (B\fc F)$ %\item $(F\cap B) \subseteq B*F.$ \item $F \subseteq (F\fc F) \subseteq (F\fc F)\fc F
   % \subseteq ... $ \item $F \subseteq (F\fc F) \subseteq F\fc (F\fc F) \subseteq ... $ \item If the scale function $\Lambda$ is null, $F\fc %  B=F$
    for any subsets $F\subset \mathcal{L}^p(I)$ and $B \subset \mathcal{L}^p(I).$
\end{enumerate}
Denote the distance between subsets of a normed space as
\[
\delta(A, C) =\inf \{\Vert g - g'\Vert: g \in A, g'\in C\}.
\]
\begin{remark}  $\delta$ does not satisfy the properties of a metric.
\end{remark}
As consequences of Proposition \ref{prop:051018}, we obtain the following estimates:
\[
\delta(F*_Tb, F*_Tb')\leq \frac{\Lambda}{ 1 -  \Lambda} \Vert b - b'\Vert_p,
\]
\[
\delta(f*_TB, f'*_TB)\leq \frac{1}{ 1 - \Lambda} \Vert f - f'\Vert_p,
\]
with the modification $\Lambda \to \Lambda^p$ for $0 < p <1$. (See Theorem \ref{prop3.2}.)
\begin{proposition}
If $F$ is bounded with bound  $ M_F$ and $B$ is bounded with bound $M_B,$ then for any $f, b\in E$:
\[
\delta(F, F*_Tb) \leq  \frac{\Lambda}{ 1 -  \Lambda} (M_F+\Vert b \Vert_p),
\]
\[
\delta(B, f*_TB) \leq  \frac{1}{ 1 -  \Lambda} (M_B+\Vert f \Vert_p),
\]
with the modification $\Lambda \to \Lambda^p$ for $0 < p <1$.
\end{proposition}
\begin{proof} Let us consider the first inequality.
If $f, f' \in F$ then
\[\inf_{f, f' \in F} \Vert f'-f*_Tb\Vert_p \leq \inf_{f\in F} \Vert f-f*_Tb\Vert_p \leq \frac{\Lambda}{ 1 -  \Lambda} (M_F+\Vert b \Vert_p),\]
according to (\ref{eq5}). The second expression can be proved analogously.
\end{proof}

\begin{remark} Let us note that $F$ and $F*_Tb$ can be approximated by taking vector sequences of scale functions $(\alpha_m)$ whose scaling factors $\Lambda_m$ tend to zero.
\end{remark}

\begin{remark}\label{rem5.3} In order to simplify the notation we will refer to $\mathcal{L}^p(I)$ as $E,$ bearing in mind that most of the results hold for other standard spaces of functions such as $\mathcal{C}(I)$ as well.
\end{remark}

Suppose that $F, B$ are compact sets of functions. Then $\mathcal{P}^1_f(B)=f*_TB$ and $\mathcal{P}^2_b(F)=F*_Tb$  are also compact sets for any $f, b\in E$. Likewise, the set
$F*_TB=\mathcal{P}(F \times B) $ is also compact.

 The Hausdorff distance between two nonempty compact sets $H, K$ is defined as
\[d_\mathcal{H}(H,K) =\max \{\sup_{x\in H} \delta(x, K), \sup_{y\in K} \delta(y, H)\}.\]

\begin{theorem}
Let us assume that $F, F', B, B'$ are nonempty compact sets with bounds $M_F, M_{F'}, M_B$, and $M_{B'}$, respectively. Then, for any $f, b \in E$,
\[d_\mathcal{H}(F, F*_Tb) \leq  \frac{\Lambda}{ 1 - \Lambda} (M_F+\Vert b \Vert_p),\]
\[d_\mathcal{H}(B, f*_TB) \leq  \frac{1}{ 1 -  \Lambda} (M_B+\Vert f \Vert_p),\]
\[d_\mathcal{H}(F, F*_TB) \leq  \frac{\Lambda}{ 1 -  \Lambda} (M_F+M_B),\]
\[d_\mathcal{H}(B, F*_TB) \leq  \frac{1}{ 1 - \Lambda} (M_B+M_F),\]
\[d_\mathcal{H}(F*_TB, F'*_TB) \leq  \frac{1}{ 1 -  \Lambda} (M_F+M_{F'}),\]
\[d_\mathcal{H}(F*_TB, F*_TB') \leq  \frac{ \Lambda}{ 1 -  \Lambda} (M_B+M_{B'}),\]
\[d_\mathcal{H}(F*_TB, F'*_TB') \leq \frac{1}{ 1 - \Lambda} ((M_F+M_{F'})+ \Lambda(M_B+M_{B'})).\]
\end{theorem}
\begin{proof} Let us prove, for instance, the first inequality.
Let $b$ be an element of $E$ and $f\in F$ then
\[
\delta(f, F*_Tb) \leq  \Vert f-f*_Tb\Vert_p \leq \frac{\Lambda}{ 1 - \Lambda} (M_F+\Vert b \Vert_p),
\]
according to (\ref{eq5}) and
\[
\delta(f*_Tb, F) \leq  \Vert f-f*_Tb\Vert_p \leq \frac{\Lambda}{ 1 -  \Lambda} (M_F+\Vert b \Vert_p).
\]
The definition of $d_\mathcal{H}$ gives the first result. (Modifications for $0<p<1.)$
%The second one is proven similarly. For the third inequality, let us think that if $f \in F$ and $b\in B$ then
%\[d(f, F*_TB) \leq  \Vert f-f*_Tb\Vert_p \leq \frac{\Lambda}{ 1 -  \Lambda} (M_F+M_B),\]
%and
%\[d(f*_Tb, F) \leq  \Vert f-f*_Tb\Vert_p \leq \frac{\Lambda}{ 1 -  \Lambda} (M_F+M_B).\]
The remaining inequalities are proven in a similar fashion.
\end{proof}

%%%%%%%%%%%%%%%%%%%%%%%%%%
\section{Partial Fractal Convolutions with the Null Function}
%%%%%%%%%%%%%%%%%%%%%%%%%%
Let us fix a partition of the nonempty interval $I = [x_0,x_N]$ and a set of scale functions satisfying $\Lambda <1$. Also recall Remark \ref{rem5.3}. The convolution of $f\in E$ or $b\in E$ with the null function $0$ defines partial operators $\cP_0^1: E \to E$ and $\cP_0^2: E \to E$ as
\begin{eqnarray}\label{eq:150615}
\mathcal{P}_0^1(b) :=\mathcal{P}(0,b)= 0*_Tb,
\end{eqnarray}
\begin{eqnarray}\label{eq:150616}
\mathcal{P}_0^2(f) :=\mathcal{P}(f,0)=f*_T0.
\end{eqnarray}
The uniqueness of the fixed point (\ref{fixp}) implies that $\mathcal{P}_0^1$ and $\mathcal{P}_0^2$ are linear operators and from inequality (\ref{eq5}) we infer that for $f=0$
\begin{equation}\label{eq:150617}
\Vert 0*_Tb \Vert_p\leq \frac{\Lambda}{1- \Lambda}\Vert b\Vert_p,
\end{equation}
and therefore
\begin{equation}\label{eq:150618}
\Vert\mathcal{P}_0^1\Vert\leq \frac{\Lambda}{1-\Lambda}.
\end{equation}
where $\Vert \cdot \Vert$ represents the operator norm with respect to $\Vert \cdot \Vert_p$ in the space $E$.

\begin{remark}
Throughout this section, results are stated in terms of the norm $\Vert \cdot \Vert_p$ with the modification $\Lambda\to\Lambda^p$ for $0<p<1$.
\end{remark}
%CUIDADO CON ESTA PARTE, CREO QUE $b=0$ DE ACUERDO CON LO ANTERIOR (DESPUÃS DE (11))\\

\begin{remark}\label{rem:110519}
 The operator $\mathcal{P}_0^1$ is contractive  if $\Lambda < 1/2,$  since
 \begin{equation} \label{eq:100519}
 \Vert \mathcal{P}_0^1 \Vert \leq \frac{\Lambda}{1- \Lambda} <1.
 \end{equation}
 As $E$ is complete, there exists a unique fixed point $\widetilde{b}$
 such that $\mathcal{P}_0^1 (\widetilde{b})=0*_T\widetilde{b}=\widetilde{b}$ and this fixed point is a function attractor. Then, according to (\ref{eq:140519}),
 \[\Vert \widetilde{b} - 0\Vert_p =\Vert 0*_T \widetilde{b} - 0*_T 0 \Vert_p \leq \frac{\Lambda}{1- \Lambda} \Vert \widetilde{b} \Vert_p < \Vert \widetilde{b} \Vert_p, \]
yielding $\widetilde{b}=0.$ In this case,  for any $b\in E$, the sequence $\{ 0*_T \cdots (0*_T(0*_Tb))\}$ tends to $0$.
\end{remark}

Bearing in mind that $b*_Tb=b$,  and using (\ref{eq:140518}), gives
\begin{equation}\label{eq:150619}
\Vert 0*_Tb - b*_Tb \Vert_p\leq \frac{1}{1-\Lambda}\Vert b\Vert_p,
\end{equation}
and, therefore,
\begin{equation}\label{eq:150620}
\Vert\mathcal{P}_0^1-I \Vert\leq \frac{1}{1-\Lambda},
\end{equation}
with $I$ denoting the identity operator on $E$.

Expression (\ref{eq:140518}) also yields that
\begin{equation}\label{eq:150621}
\Vert f*_T0-0*_T0 \Vert_p\leq \frac{1}{1- \Lambda}\Vert f\Vert_p,
\end{equation}
and therefore
\begin{equation}\label{eq:150622}
\Vert\mathcal{P}_0^2\Vert\leq \frac{1}{1- \Lambda}.
\end{equation}
Thus, both operators $\mathcal{P}_0^1$ and  $\mathcal{P}_0^2$ are linear and bounded, hence continuous on $E$.\\
Moreover, by (\ref{eq:140519})
\begin{equation}\label{eq:150623}
\Vert f*_T0 -f *_T f \Vert_p\leq \frac{\Lambda}{1- \Lambda}\Vert f\Vert_p,
\end{equation}
and
\begin{equation}\label{eq:150624}
\Vert\mathcal{P}_0^2-I \Vert\leq \frac{\Lambda}{1-\Lambda}.
\end{equation}
According to (\ref{fixp}),
\begin{equation}\label{eq:140522}
\Vert f *_T 0 - f \Vert_p \leq \Lambda \Vert f *_T 0  \Vert_p,\end{equation}
and thus
\begin{equation}\label{eq:160615}
\Vert\mathcal{P}_0^2-I \Vert\leq  \Lambda \Vert\mathcal{P}_0^2 \Vert.
\end{equation}

\subsection{Further Properties of  $\mathcal{P}_0^1$}

Here, we consider some additional properties of $\mathcal{P}_0^1$ and, in particular, the property of generating fractal bases for $E$.
\begin{proposition}\label{prop:160617}
If $\alpha_n$ is non-null almost everywhere for any $n$ then $\mathcal{P}_0^1$ is injective.
\end{proposition}
\begin{proof} Suppose $0*_Tb=0$. By the fixed point equation (\ref{fixp}) and for a $t\in I$ such that $L_n(t) \in L_n(I)$, we have
\[\alpha_n(t)b(t)=0,\]
and thus $b=0.$
\end{proof}

%Let us consider now $E$ a Banach space of functions on $I$ with respect to the uniform norm. \\
%EL SIGUIENTE RESULTADO HAY QUE REPASARLO

\begin{proposition} \label{prop:140521}
Let $b \in E$ and $\alpha_n(t)=\pm \Lambda$, for all $n$ and $t\in I$. Then,
\begin{equation}\label{eq:140520}
\Vert 0*_Tb \Vert_p= \Lambda  \Vert 0*_Tb -b\Vert_p.\end{equation}
If the maps $\alpha_n$ are constant functions and  $b \in E$,
\begin{equation} \label{eq:110520}
\Vert 0*_Tb \Vert_p \leq \Lambda \Vert 0*_Tb -b\Vert_p,\end{equation}
and therefrom
\begin{equation}\label{eq:160616}
\Vert\mathcal{P}_0^1 \Vert\leq  \Lambda \Vert\mathcal{P}_0^1-I \Vert.
\end{equation}
\end{proposition}

\begin{proof}  The expressions (\ref{eq:140520}) and (\ref{eq:110520}) are consequence of Equation (\ref{fixp}).
%For the case $p<\infty$,  according to the fixed point equation,
%\[\Vert 0*_Tb\Vert_p^p= \sum_{n=1}^N \int_{I_n} \Lambda^p \vert (0*_Tb-b)(t)\vert^p,\]
%and thus
%\[\Vert 0*_Tb\Vert_p^p= \Lambda^p \Vert 0*_Tb-b\Vert_p^p.\]
%The case $p=\infty$ is deduced from the same expression ...
%\[\sup\{\vert 0*_Tb(t)\vert: t\in I\}=\sup\{\vert \alpha_n(t) \vert \vert 0*_Tb(t)-b(t) \vert: t\in I, n=1,2,\ldots,N\}\]
%and therefrom the result.
\end{proof}

%ESTO NO SE DE DONDE VIENE:
%\[\Vert\mathcal{P}_0^1 \Vert\leq  \Lambda.\]

\begin{proposition}\label{prop:160618}
Suppose that $\alpha_n(t)=\pm \Lambda$, for all $n$ and $t\in I$. Then $\mathcal{P}_0^1$ has closed range.
\end{proposition}
\begin{proof} The previous result (\ref{eq:140520}) provides the inequality
\[\Lambda (\Vert b\Vert_p - \Vert 0*_Tb\Vert_p) \leq \Vert 0*_Tb\Vert_p,\]
and thus
\begin{equation}\label{eq:150518}
\Vert b\Vert_p \leq \frac{1+\Lambda }{ \Lambda } \Vert 0*_Tb\Vert_p= \frac{1+\Lambda }{\Lambda}\Vert \mathcal{P}_0^1(b)\Vert_p,
\end{equation}
if $\Lambda\neq 0$.
%The rest is similar to the proof of Theorem 3.5 of \cite{caot}. \\
Given a sequence $\mathcal{P}_0^1 (b_m)$ converging to $\overline{b}$, the linearity of $\mathcal{P}_0^1$ implies that
\[\Vert b_m -b_n \Vert_p \leq \frac{1+\Lambda }{\Lambda}\Vert \mathcal{P}_0^1(b_m)-\mathcal{P}_0^1(b_n)\Vert_p.\]
Hence, $(b_m)$ is a Cauchy sequence and $b_m \rightarrow b.$ The continuity of $\mathcal{P}_0^1$ implies that $\overline{b}=\mathcal{P}_0^1(b).$\\
If $\Lambda$ is zero then it follows from (\ref{eq:160616}) that $\mathcal{P}_0^1$ is the null operator.
\end{proof}

\begin{figure}[h]
\includegraphics[angle=0, width=1\textwidth]{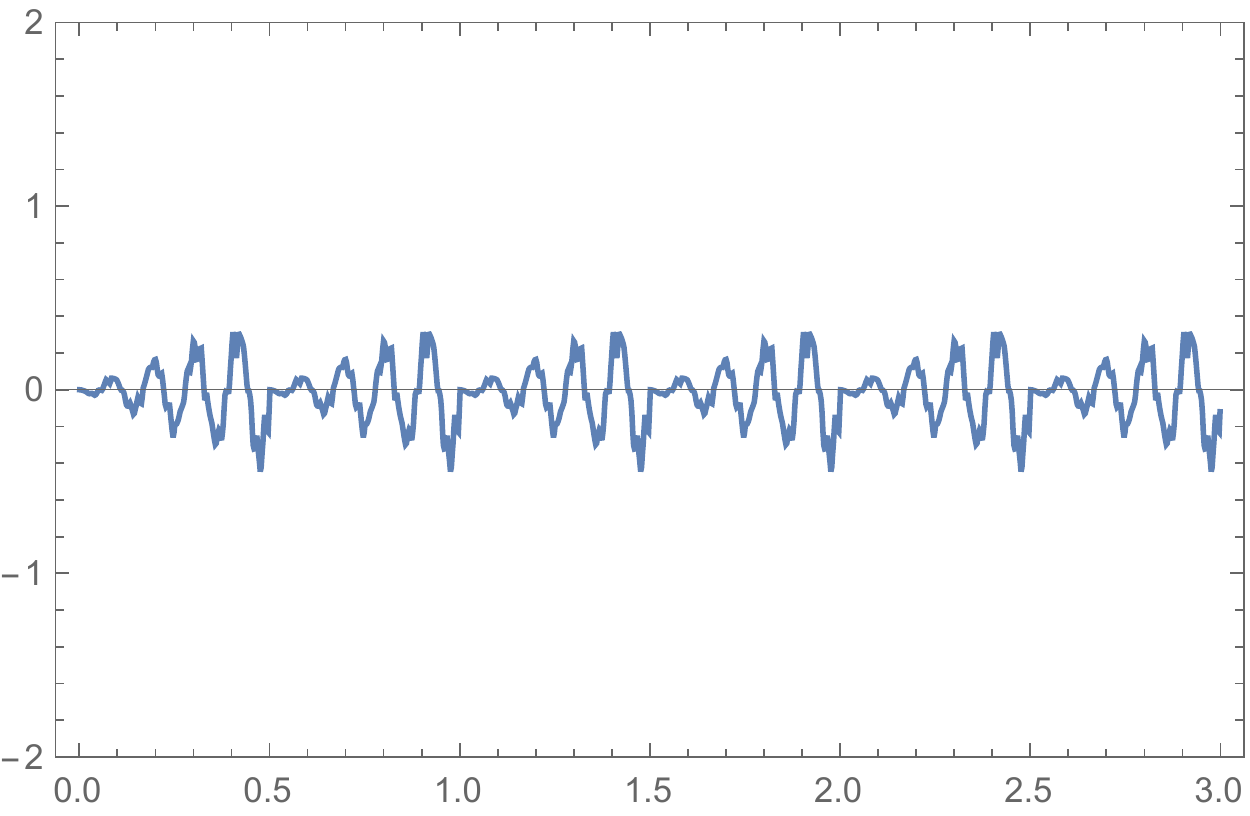} \caption{\small Graph of the fractal convolution
$0*_T\sin(3 \pi x)$, with respect to a uniform  partition with $N=6$,
and  scale functions $\alpha_n(x)=x/8$ for $n=1, 2, \ldots, 6$ in the interval $I=[0,3]$.}
\end{figure}

\subsection{Fractal Convolution and Schauder Bases}
Here, we consider the generation of fractal-like bases via the operator $\cP_0^1$. To this end, we refer to, e.g., \cite{10, 19}, for the definitions of Schauder basis, Riesz sequences and bases, and Bessel sequences. As it is well-known, these systems provide spanning families for Banach and Hilbert spaces.

\begin{definition}
A sequence $(x_m)$ in a Banach space $(E, \Vert \cdot \Vert)$ is called a \emph{Schauder sequence} if it is a Schauder basis for the closed span of $(x_m)$ (usually denoted by $[x_m].$)
\end{definition}

\begin{proposition}\label{prop:160619}
Suppose that $\alpha_n(t)=\pm \Lambda\neq 0$, for all $n$ and for all $t\in I$. If $(b_m)$ is a Schauder basis of $E$ then $(0*_Tb_m)$ is a Schauder sequence of $E$.
\end{proposition}
\begin{proof} Propositions \ref{prop:160617} and \ref{prop:160618} imply that $\mathcal{P}_0^1$ is a topological isomorphism on its range and thus it transforms bases into bases.
\end{proof}
\begin{proposition}\label{prop:160620}
If $\Lambda < 1/2$ and $(b_m)$ is a Schauder basis of $E$, then $(b_m - 0*_Tb_m)$ is a Schauder basis of $E$. Furthermore, if the $\alpha_n$ are constant functions and $(b_m)$ is a bounded Schauder basis of $E$ then $(0*_Tb_m)$ is also bounded.
\end{proposition}
\begin{proof} In this case, $\Vert \mathcal{P}_0^1 \Vert < 1$ (Remark \ref{rem:110519}) and $I- \mathcal{P}_0^1$ is a topological isomorphism preserving bases.\\
If $(b_m)$ is bounded, there exist by definition positive constants $k$ and $K$  such that for all $m$,
\[k\leq \Vert b_m \Vert_p \leq K.\]
For the fractal basis $(0\fc b_m)$, we have by (\ref{eq:140518}),
\begin{equation}\label{eq:140527}
\Vert b_m - 0*_T b_m \Vert_p \leq   \frac{1}{1-\Lambda}\Vert b_m \Vert_p \leq \frac{K}{1-\Lambda}.\end{equation}
For the lower bound, one obtains, using (\ref{eq:110520}),
\[\Vert b_m - 0*_T b_m \Vert_p \geq \Vert b_m \Vert_p  - \Vert 0*_T b_m \Vert_p \geq \Vert b_m \Vert_p - \Lambda \Vert 0*_T b_m -b_m \Vert_p, \]
and therefore,
\begin{equation}\label{eq:160624}
(1+\Lambda) \Vert b_m - 0*_T b_m \Vert_p \geq  \Vert b_m \Vert_p \geq k.\end{equation}
\end{proof}
%The next results hold in the general space $E$.

\begin{theorem}\label{prop:16062}
For any $\Lambda < 1$, $I-\mathcal{P}_0^1$ is injective and has closed range.
\end{theorem}
\begin{proof}
This follows from inequality (\ref{eq:160624}) and the proof of Proposition \ref{prop:160618}.
\end{proof}
As a consequence of the above results, we obtain
\begin{corollary}\label{prop:160624}
If $(b_m)$ is a Schauder basis then $(b_m-0*_Tb_m)$ is a Schauder sequence.
\end{corollary}

\subsection{Fractal Convolution and the Hilbert Space $\mathcal{L}^2(I)$}

The existence of an inner product generates new ways of constructing systems of fractal functions.

\begin{proposition}\label{prop:160621}
Suppose $(b_m)$ is an orthonormal basis of $\mathcal{L}^2(I)$. Then $(0*_T b_m)$ is a Bessel sequence in $\mathcal{L}^2(I)$.
\end{proposition}
\begin{proof}
The proof is similar to that of \cite[Proposition 2.18]{10}.
\end{proof}

\begin{proposition}\label{prop:160622}
Assume that $\alpha_n(t)=\pm \Lambda \neq 0$, for all $n$ and $t\in I,$ and that
 $(b_m)$ is an orthonormal basis of $\mathcal{L}^2(I)$. Then $(0*_Tb_m)$ is a Riesz sequence in $\mathcal{L}^2(I)$.
\end{proposition}
\begin{proof}
The proof follows arguments similar to those given for Proposition 2.21 of \cite{10}, using inequality (\ref{eq:150518}).
\end{proof}

\begin{proposition}\label{prop:160623}
If $(b_m)$ is an orthonormal basis of $\mathcal{L}^2(I)$ and $\Lambda < 1/2$ then $(b_m-0*_Tb_m)$ is a Riesz basis of $\mathcal{L}^2(I)$.
\end{proposition}
\begin{proof}
In this case $\Vert \mathcal{P}_0^1 \Vert < 1$ (Remark \ref{rem:110519}) and thus $I- \mathcal{P}_0^1$ is a topological isomorphism on its range. Consequently $(b_m-0*_Tb_m)$  is equivalent to an orthonormal basis and thus a Riesz basis.
\end{proof}

%\begin{figure}[h] %\includegraphics[angle=0, width=1\textwidth]{fig_convoluciones_parabolax0_alfa_fun.pdf} %\caption{\small Graph of the fractal convolution
 %$(t-1.5)^2*_T0$, with respect to a uniform  partition with $N=6$,
%and  scale function $\Lambda(t)=t/8$ in the interval $I=[0,3]$.} %\end{figure}

\subsection{Properties of  $\mathcal{P}_0^2$}
Let us consider now the properties of the operator $$\mathcal{P}_0^2(f)=f*_T0$$ in the general case $E=\mathcal{L}^p(I)$, $1\leq p \leq \infty$.

If $f*_T0=0$ then (\ref{eq:140522}) implies that $f=0$. Hence, the operator $\mathcal{P}_0^2$ is injective. By the same inequality,
\[\Vert f \Vert_p - \Vert f*_T0 \Vert_p \leq \Lambda  \Vert f*_T0 \Vert_p, \]
and
\begin{equation}\label{eq:140523}
\Vert f \Vert_p  \leq (1+ \Lambda)  \Vert f*_T0 \Vert_p.
\end{equation}
The latter inequality implies that  $\mathcal{P}_0^2$ has closed range. (See also the proof of Proposition \ref{prop:160618}.)

Thus the construction of fractal bases via the operator $I-\mathcal{P}_0^1$ is applicable to $\mathcal{P}_0^2$ without any hypotheses on the scale functions. For instance, we have the following result.
\begin{proposition} If $(f_m)$ is a Schauder basis of $E$, then $(f_m*_T0)$ is a Schauder sequence in $E$.
\end{proposition}
\begin{proposition} \label{prop:140526}
If $(f_m)$ is a Schauder basis of $E$ and $ \Lambda  <1/2$ then $(f_m*_T0)$ is a Schauder basis. If  $(f_m)$ is bounded then the fractal basis $(f_m*_T 0)$ is also bounded.
\end{proposition}
\begin{proof}
By (\ref{eq:150624}), $\Vert I-  \mathcal{P}_0^2 \Vert <1$ and therefore $\mathcal{P}_0^2$ is a topological isomorphism which preserves bases. Using (\ref{eq:140523}), the following inequalities
\[\frac{\Vert f \Vert_p }{1+\Lambda} \leq \Vert f*_T0 \Vert_p \leq \Vert \mathcal{P}_0^2 \Vert \Vert f \Vert_p\]
imply the boundedness of the new basis.
\end{proof}
%For the case $p=2$ we have some results similar to those of $\mathcal{P}_0^1$. \\

We now refine Propositions \ref{prop:160621} and \ref{prop:160622} in the case $p=2.$

\begin{proposition}
If $(f_m)$ is a Bessel sequence in $\mathcal{L}^2(I)$ then $(f_m *_T 0)$ is also a Bessel sequence.
\end{proposition}
\begin{proof}
Suppose $(f_m)$ is a Bessel sequence. By definition, there exists a constant $B>0$ such that for any $g \in E$
\[
\sum_m \vert \inn{g}{f_m} \vert^2 \leq B \Vert g \Vert^2_2.
\]
For the family $(f_m *_T 0)$ and an $f \in E$ we have
\[
\sum_m \vert \inn{f}{f_m *_T 0}\vert^2 = \sum \vert \inn{f}{\mathcal{P}_0^2(f_m)} \vert^2 = \sum_m \vert \inn{(\mathcal{P}_0^2)^*(f)}{f_m}\vert^2,
\]
where $(\mathcal{P}_0^2)^*$ is the adjoint operator of $\mathcal{P}_0^2.$ Applying the Bessel property of $(f_m)$ yields
\[
\sum_m \vert \inn{f}{f_m *_T 0}\vert^2 \leq B \Vert (\mathcal{P}_0^2)^* \Vert^2 \Vert f \Vert^2_2,
\]
proving the statement.
\end{proof}

\begin{theorem}
If $(f_m)$ is a Riesz sequence in $\mathcal{L}^2(I)$, then $(f_m *_T 0)$ is also a Riesz sequence.
\end{theorem}
\begin{proof}
If $(f_m)$ is a Riesz sequence then it is a basis for $[f_m]$. The operator $\mathcal{P}_0^2$ is a topological isomorphism from $[f_m]$ onto $[f_m *_T 0],$ and thus it preserves bases.
\end{proof}

As a consequence of Proposition \ref{prop:140526} we obtain the next result.

\begin{corollary}  Suppose $(f_m)$ is an  orthonormal basis of $\mathcal{L}^2(I)$ and $\Lambda <1/2$. Then $(f_m*_T0)$  is a bounded Riesz basis of $\mathcal{L}^2(I)$.
\end{corollary}

%\begin{itemize}
%\item If $(f_m)$ is Bessel sequence  then $(f_m*_T0)$ also is.

%\item If $(f_m)$ is Riesz sequence  then $(f_m*_T0)$ also is.
%\item
%\end{itemize}

\begin{figure}[h]
\includegraphics[angle=0, width=1\textwidth]{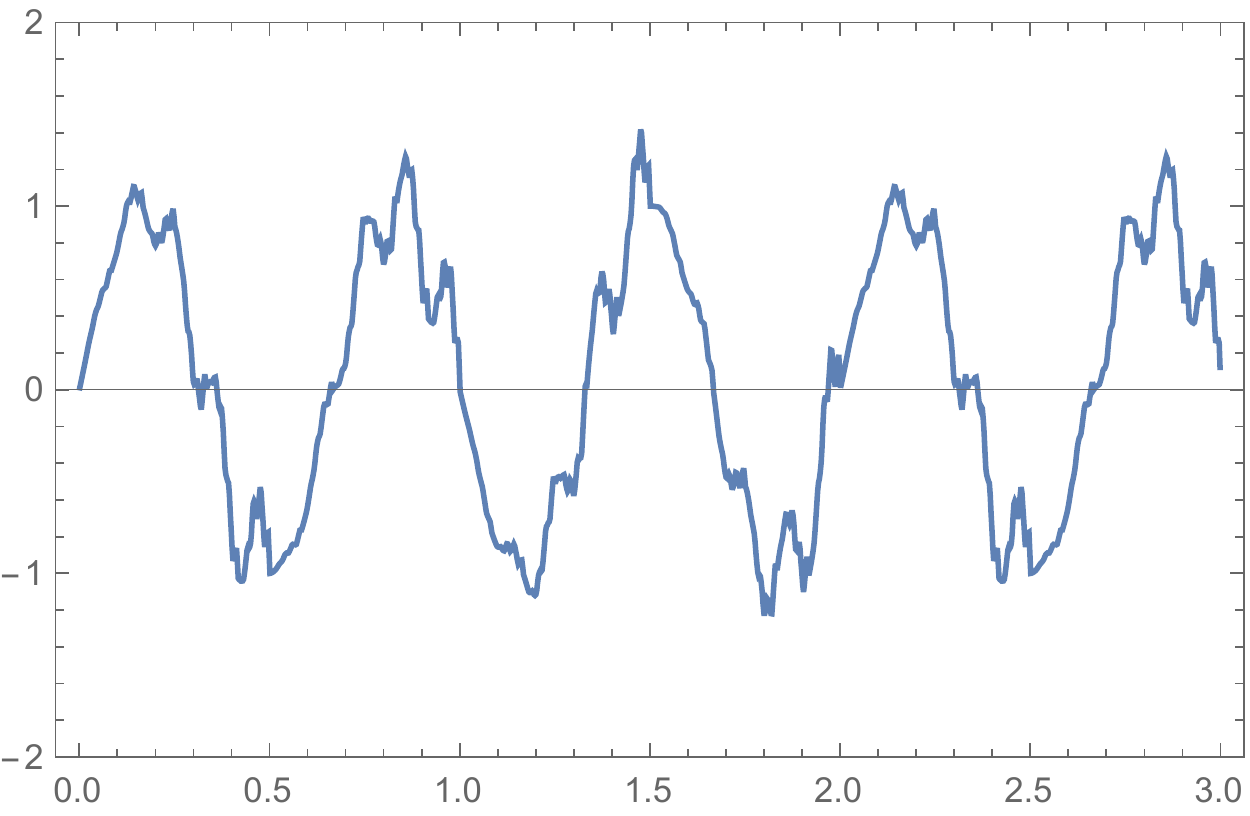} \caption{\small Graph of the fractal convolution
 $\sin(3 \pi x)*_T0$, with respect to a uniform  partition of the interval $I=[0,3]$ with $N=6$,
and  scale functions $\alpha_n(x)=x/8$ for $n=1, 2, \ldots, 6$.}
\end{figure}

%%%%%%%%%%%%%%%%%%%%%%%%%%
\section{Fractal Convolutions and Frames}
%%%%%%%%%%%%%%%%%%%%%%%%%%

In this section, we investigate the relation between fractal convolution and frames in separable Hilbert spaces.

\begin{definition} Let $H$ be a separable Hilbert space. A sequence $(f_m)\subset H$ is called a \emph{frame for $H$} if there exist positive constants $A, B$, called \emph{frame bounds}, such that
\[
A \Vert f \Vert^2 \leq \sum_m \vert \inn{f}{f_m} \vert \leq B \Vert f \Vert^2, \quad \forall\,f\in H.
\]
\end{definition}
\noindent
The next result can be found in reference \cite{20}.

\begin{theorem}\label{th:140528}
Let  $(f_m)$ be a frame with frame bounds $A,B> 0$. Any family $(g_m)$ such that
\[
R=\sum_m \Vert f_m -g_m \Vert^2 <A,
\]
constitutes a frame with frame bounds $A\left(1- \sqrt{\frac{R}{A}}\right)^2$ and  $B\left(1+ \sqrt{\frac{R}{B}}\right)^2.$
\end{theorem}

\begin{theorem}
Let  $(f_m)$ be a frame in a separable Hilbert space H with frame bounds $A, B>0$. Choose a sequence of scale functions $\{\alpha_n^m\}$ in the definition of $(f_m*_T0)$ such that
\[
R=\sum_m \left(\frac{\Lambda_m }{1-\Lambda_m }\right)^2 \Vert f_m \Vert_2^2 <A,
\]
where $\Lambda_m=\max\{\Vert \alpha_n^m \Vert_\infty : n\in \N_N\}$. Then $(f_m*_T0)$ is a frame in $H$ with frame bounds  $A\left(1- \sqrt{\frac{R}{A}}\right)^2$ and  $B\left(1+ \sqrt{\frac{R}{B}}\right)^2.$
\end{theorem}
\begin{proof}
The statement follows from inequality (\ref{eq5})
\[\Vert f_m*_T0 -f_m \Vert_2 \leq \frac{\Lambda_m  }{1-\Lambda_m  } \Vert f_m \Vert_2\]
and the previous theorem.
\end{proof}

\begin{remark} If $(p_m)$ is an orthonormal basis of $H$, it is a frame in $H$ with frame bounds $A=B=1$. In this case, the fractal analogue would be a frame in $H$ with frame bounds $(1-\sqrt{R})^2$ and  $(1+\sqrt{R})^2$. Thus, it suffices to choose suitable $\Lambda_m \in \mathcal{O}(1/m)$ in order to construct a fractal frame in $H$.
\end{remark}

\begin{remark} The union of an orthonormal basis of $H$ and a Bessel sequence in $H$ is a frame. Hence, $(p_m)\cup (p_m*_T0)$ is a frame in $H$ and this fact enables the
simultaneous use of classical as well as fractal maps in order to span and approximate a given function in $H$.
\end{remark}

\begin{remark} By (\ref{eq:140527}) and Theorem \ref{th:140528} it suffices to choose the $\Lambda_m$ such that
\[\sum_m \left(\frac{\Vert b_m \Vert_2}{1-\Lambda_m }\right)^2 <A\]
in order to obtain a fractal frame $(0*_T b_m)$ from a frame $(b_m)$ in $H$.
\end{remark}

%Final consideration: If one wishes to convolve a basis with a function different from zero,  for instance $b \in E$ one can construct $(f_m*_Tb)$ with %a
%sequence of scale functions such that
%\[\sum_{m=1}^\infty \Lambda_m  < \infty.\]
%Arguments similar to those provided in \cite{quest} would prove that the $(f_m*_Tb)$ would be ....

%\begin{figure}[h]
%\includegraphics[angle=0, width=1\textwidth]{fig_convoluciones_0xb_alfa_fun.pdf} \caption{\small Graph of the fractal convolution
% $0*_T\sin(3 \pi x)$, with respect to a uniform  partition with $N=6$,
%and  scale functions $\alpha_n(x)=x/8$ for $n=1, 2, \ldots, 6$ in the interval $I=[0,3]$.}
%\end{figure}

\end{document}